\documentclass[11pt]{amsart}
\usepackage[T1]{fontenc}
\usepackage[utf8]{inputenc}
\usepackage[
  backend=biber,
  style=numeric,
  sorting=none,
  giveninits=false,
  doi=true,
  isbn=true,
  url=true,
  eprint=true
]{biblatex}
\usepackage{lmodern}
\usepackage{microtype}
\usepackage{tabularx}
\usepackage{array}
\usepackage{amsmath,amssymb,amsfonts,amsthm}
\usepackage{mathtools}
\usepackage{enumitem}
\usepackage{tikz-cd}
\usepackage{geometry}
\usepackage{float}

\usepackage[colorlinks=true,linkcolor=blue,citecolor=blue,urlcolor=blue]{hyperref}
\usepackage[nameinlink,noabbrev]{cleveref}

\hypersetup{
  pdftitle={From Finite-Node Conifold Geometry to BPS Structures III},
  pdfauthor={Abdul Rahman}
}

\numberwithin{equation}{section}

\theoremstyle{plain}
\newtheorem{theorem}{Theorem}[section]
\newtheorem{proposition}[theorem]{Proposition}

\newtheorem{corollary}[theorem]{Corollary}

\newtheorem{hypothesis}[theorem]{Hypothesis}
\theoremstyle{definition}
\newtheorem{definition}[theorem]{Definition}
\newtheorem{example}[theorem]{Example}
\newtheorem{convention}[theorem]{Convention}

\theoremstyle{remark}
\newtheorem{remark}{Remark}[section]

\makeatletter



\makeatother

\newcommand{\Q}{\mathbb{Q}}

\newcommand{\RHom}{\mathbf{R}\!\operatorname{Hom}}

\newcommand{\Perv}{\mathrm{Perv}}
\newcommand{\Hom}{\mathrm{Hom}}
\newcommand{\Ext}{\mathrm{Ext}}

\title[From Finite-Node Conifold Geometry to BPS Structures III]
{From Finite-Node Conifold Geometry to BPS Structures III: Mediated Triangle Transport and Graded Interaction Data}
\author{Abdul Rahman}
\thanks{Email: arahman@alum.howard.edu}
 % edit

\begin{document}

\begin{abstract}
In previous work, we extracted from a finite-node conifold degeneration the state-data package $A_\Sigma=(V_\Sigma,E_\Sigma,c_\Sigma)$ and then constructed the support-level interaction package encoded by a binary incidence structure and finite quiver-theoretic skeleton \cite{RahmanQuiverDataI,RahmanQuiverDataII}. The present paper introduces the next layer: a graded pairwise interaction package refining binary support. Since the support matrix records where a mediated channel is present, but not its derived size, cohomological degree, or exact-triangle behavior, we introduce \emph{mediated triangle transport} (MTT). An MTT datum combines bulk-mediated schober transport, localized probes, corrected-extension shadow compatibility, and derived interaction profunctors. For each ordered pair $(i,j)$, it produces $\mathbb T_{ij}(X,Y):=\RHom_{\mathcal C_{p_j}}(\Psi_j\Phi_i(X),Y)$ and the probe interaction complex $\mathsf H_{ij}:=\mathbb T_{ij}(L_i,L_j)=\RHom_{\mathcal C_{p_j}}(\Psi_j\Phi_i(L_i),L_j)$. We prove exactness and long exact interaction sequences, isolate a triangle-visible nonvanishing criterion, and formulate a conditional bridge theorem showing that supported channels yield nontrivial pairwise interaction complexes under the stated probe, content, and detector hypotheses. Under a bounded Hom-finite convention, the cohomology of $\mathsf H_{ij}$ defines $P_{ij}(q)=\sum_m \dim H^m(\mathsf H_{ij})q^m$, and these polynomials assemble into $I_\Sigma^{\mathrm{gr}}$. Thus $(A_\Sigma,I_\Sigma^{(0/1)},I_\Sigma^{\mathrm{gr}})$ provides the first graded interaction input for later stability, BPS, and wall-crossing theory.
\end{abstract}
\maketitle
\tableofcontents

%----------------------------------------------------
\section{Introduction}

Let
\[
\pi:X\to\Delta
\]
be a one-parameter degeneration whose central fiber
\[
X_0:=\pi^{-1}(0)
\]
has singular locus
\[
\Sigma=\{p_1,\dots,p_r\}\subset X_0
\]
consisting of finitely many ordinary double points.  The finite-node setting naturally separates the geometry into a common bulk sector and finitely many localized node sectors.  In the present program this bulk/localized structure has been extracted in three compatible forms: the corrected perverse extension \cite{RahmanPerverseNearbyCycles}, its mixed-Hodge-module refinement \cite{RahmanMixedHodgeModules}, and the finite-node schober package \cite{RahmanSchoberPaper,RahmanMultiNodeSchoberPaper}.

The present paper continues the sequence
\[
\text{corrected extension}
\longrightarrow
A_\Sigma
\longrightarrow
I_\Sigma^{(0/1)}
\longrightarrow
I_\Sigma^{\mathrm{gr}}.
\]
The first finite algebraic layer,
\[
A_\Sigma=(V_\Sigma,E_\Sigma,c_\Sigma),
\]
was constructed in \cite{RahmanQuiverDataI}.  The next layer, the support-level incidence law and finite quiver-theoretic skeleton, was constructed in \cite{RahmanQuiverDataII}.  This paper constructs the first graded refinement of that support law.

The need for this refinement is structural.  The binary matrix \(I_\Sigma^{(0/1)}\) records whether a mediated channel is present, but it does not record how the channel is realized in a derived category, in which cohomological degrees it contributes, or how localized information is transported through the bulk.  For later stability, BPS, and wall-crossing applications, support is therefore too coarse: one needs a pairwise interaction object before one can ask for stability data or BPS indices.

This distinguishes the present construction from approaches in which one begins with an already available quiver, charge lattice, stability condition, or BPS sector.  In the existing quiver and BPS literature, such data support rich wall-crossing and mutation formalisms \cite{SeidelThomas,AlimCecottiCordovaEspahbodiRastogiVafa_BPSQuivers,AlimCecottiCordovaEspahbodiRastogiVafa_N2Quivers,Denef_QuantumQuivers,Cecotti_QuiverBPS}.  Here the question is earlier: starting from the finite-node conifold geometry, the corrected-extension package, and the associated schober transport data, can one extract the first derived pairwise interaction package that a later BPS theory may inherit?

The answer proposed here is \emph{mediated triangle transport}.  For each ordered pair \((i,j)\), the finite-node schober datum supplies exact functors
\[
\Phi_i:\mathcal C_{p_i}\to\mathcal C_{\mathrm{bulk}},
\qquad
\Psi_j:\mathcal C_{\mathrm{bulk}}\to\mathcal C_{p_j},
\]
and hence a bulk-mediated composite
\[
\Psi_j\Phi_i:\mathcal C_{p_i}\to\mathcal C_{p_j}.
\]
Given localized probes \(L_i\in\mathcal C_{p_i}\) and \(L_j\in\mathcal C_{p_j}\), we define the pairwise interaction complex
\[
\mathsf H_{ij}
:=
\RHom_{\mathcal C_{p_j}}(\Psi_j\Phi_i(L_i),L_j).
\]
Thus support-level incidence is refined not by assigning scalar weights by hand, but by producing derived interaction objects from the bulk/localized transport structure.

The paper remains a bridge paper.  It does not construct BPS indices, central charges, stability conditions, or wall-crossing automorphisms.  Its output is the graded interaction package that such later structures may consume.

\subsection{Input from previous work}

We use three inputs.

First, the corrected finite-node extension package identifies the point-supported correction terms attached to the nodes \cite{RahmanPerverseNearbyCycles}.  Its mixed-Hodge-module refinement lifts the same architecture to the Hodge-theoretic setting and retains the finite localized quotient structure \cite{RahmanMixedHodgeModules}.

Second, the finite-node schober formalism supplies a bulk category
\[
\mathcal C_{\mathrm{bulk}},
\]
localized categories
\[
\mathcal C_{p_i},
\qquad 1\le i\le r,
\]
and exact attachment functors
\[
\Phi_i:\mathcal C_{p_i}\to\mathcal C_{\mathrm{bulk}},
\qquad
\Psi_i:\mathcal C_{\mathrm{bulk}}\to\mathcal C_{p_i},
\]
compatible with the corrected perverse shadow \cite{RahmanSchoberPaper,RahmanMultiNodeSchoberPaper}.

Third, the first two papers of this sequence provide the finite algebraic and support-level data. The paper  \cite{RahmanQuiverDataI} constructs
\[
A_\Sigma=(V_\Sigma,E_\Sigma,c_\Sigma),
\]
where
\[
V_\Sigma=\{v_1,\dots,v_r\},\qquad
E_\Sigma=\Ext^1_{\Perv(X_0;\Q)}(Q_\Sigma,IC_{X_0}),
\]
and \(c_\Sigma\in\Q^r\) records the corrected global class in the nodewise basis \cite{RahmanQuiverDataI}. The paper  \cite{RahmanQuiverDataII} constructs the binary support law
\[
I_\Sigma^{(0/1)}
\]
and the associated finite quiver-theoretic skeleton \cite{RahmanQuiverDataII}.  The present paper takes these structures as fixed input.

\subsection{Related work}

The perverse-sheaf background is rooted in the foundational theory of Beilinson--Bernstein--Deligne \cite{BBD} and in later descriptions of perverse objects and their local algebraic models \cite{MacPhersonVilonen1986,GMV1996,KS,DimcaSheaves,Schurmann}.  The Hodge-theoretic background includes the theory of degenerations and limiting mixed Hodge structures \cite{Schmid,SteenbrinkLimits,DeligneDegeneration,VoisinHodgeTheoryI,VoisinHodgeTheoryII}, together with Saito's mixed-Hodge-module formalism \cite{SaitoMHM,SaitoDuality,SaitoYoungGuide}.  The conifold setting itself belongs to the classical study of isolated hypersurface singularities and Calabi--Yau degenerations \cite{MilnorSingularPoints,ClemensDegeneration,StromingerConifold,GreeneMorrisonStrominger,CollinsConifoldIntro}.

The categorical side is motivated by the theory of perverse schobers and categorical transport \cite{KapranovSchechtman,AnnoLogvinenko,KatzarkovPanditSpaide,KosekiOuchi}.  In this paper, the relevant schober input is the finite-node bulk/localized transport package of \cite{RahmanSchoberPaper,RahmanMultiNodeSchoberPaper}, whose shadow is compatible with the corrected perverse and mixed-Hodge-module extension packages of \cite{RahmanPerverseNearbyCycles,RahmanMixedHodgeModules}.

The quiver, BPS, and wall-crossing literature shows that once suitable quiver, charge, stability, or BPS data are available, one can formulate nontrivial chamber and wall-crossing structures \cite{SeidelThomas,AlimCecottiCordovaEspahbodiRastogiVafa_BPSQuivers,AlimCecottiCordovaEspahbodiRastogiVafa_N2Quivers,Denef_QuantumQuivers,Cecotti_QuiverBPS}.  The present work addresses the preceding extraction problem: how much pairwise interaction data can be obtained before choosing a stability condition or postulating a BPS sector?

\subsection{From support to mediated triangle transport}

The support-level package of \cite{RahmanQuiverDataII} records where interaction is allowed.  The present paper replaces support-only data by derived pairwise interaction objects:
\[
\text{support}
\longrightarrow
\text{derived pairwise interaction}
\longrightarrow
\text{graded interaction shadow}.
\]

For each ordered pair \((i,j)\), define
\[
\mathbb T_{ij}(X,Y)
:=
\RHom_{\mathcal C_{p_j}}(\Psi_j\Phi_i(X),Y),
\qquad
X\in\mathcal C_{p_i},\quad Y\in\mathcal C_{p_j}.
\]
This is contravariant in \(X\) and covariant in \(Y\), hence has the form of a derived interaction profunctor
\[
\mathbb T_{ij}:
\mathcal C_{p_i}^{op}\times \mathcal C_{p_j}
\to
D(\mathrm{Vect})
\]
in the sense of \cite{BenabouProfunctors}.  Evaluating on localized probes gives
\[
\mathsf H_{ij}
:=
\mathbb T_{ij}(L_i,L_j)
=
\RHom_{\mathcal C_{p_j}}(\Psi_j\Phi_i(L_i),L_j).
\]
These complexes are the basic pairwise interaction objects of the paper.

\begin{remark}
A profunctor from \(\mathcal A\) to \(\mathcal B\) is a bifunctor
\[
\mathcal A^{op}\times\mathcal B\to\mathcal V
\]
for a target category \(\mathcal V\) \cite{BenabouProfunctors,LoregianCoendCalculus}.  Here \(\mathcal V=D(\mathrm{Vect})\), and the derived Hom construction above supplies the required bifunctor.  Appendix~\ref{app:profunctors} recalls only the profunctorial facts used in the paper.
\end{remark}

\subsection{Finiteness convention}

The graded decategorification step requires a mild finiteness convention.  Throughout the parts of the paper where \(P_{ij}(q)\) is defined, we assume that the relevant localized categories are \(k\)-linear and Hom-finite on the probe-generated objects under consideration.  In particular, the complexes
\[
\mathsf H_{ij}
=
\RHom_{\mathcal C_{p_j}}(\Psi_j\Phi_i(L_i),L_j)
\]
are assumed to have finite-dimensional cohomology and to be bounded, or else locally finite.  In the bounded Hom-finite case,
\[
P_{ij}(q)=\sum_m \dim H^m(\mathsf H_{ij})q^m
\]
is a Laurent polynomial.  In the locally finite case, the same expression is interpreted as a formal graded Poincare series.

\subsection{Main theorem stack}

We now summarize the results proved in the paper.

\begin{theorem}[obstruction to the naive extension-only bridge]
\label{thm:intro-obstruction}
Let
\[
Q_\Sigma=\bigoplus_{k=1}^r Q_k,
\qquad
Q_k=i_{k*}\Q_{\{p_k\}},
\]
be the corrected finite-node quotient on the shadow side.  The extension-only route built from the localized objects \(Q_k\) and their nodewise extension data is intrinsically local.  It does not, by itself, produce the bulk-mediated off-diagonal pairwise interaction layer needed for later stability or BPS theory.
\end{theorem}

\begin{definition}[mediated triangle transport datum]
\label{def:intro-mtt}
A mediated triangle transport datum consists of:
\begin{enumerate}
\item a bulk triangulated category \(\mathcal C_{\mathrm{bulk}}\);
\item localized triangulated categories \(\mathcal C_{p_i}\), \(1\le i\le r\);
\item exact functors
\[
\Phi_i:\mathcal C_{p_i}\to\mathcal C_{\mathrm{bulk}},
\qquad
\Psi_i:\mathcal C_{\mathrm{bulk}}\to\mathcal C_{p_i};
\]
\item localized probes \(L_i\in\mathcal C_{p_i}\);
\item compatibility with the corrected-extension shadow package of \cite{RahmanPerverseNearbyCycles,RahmanMixedHodgeModules};
\item derived interaction profunctors
\[
\mathbb T_{ij}(X,Y)
:=
\RHom_{\mathcal C_{p_j}}(\Psi_j\Phi_i(X),Y).
\]
\end{enumerate}
\end{definition}

The corresponding probe evaluation is
\[
\mathsf H_{ij}
:=
\mathbb T_{ij}(L_i,L_j)
=
\RHom_{\mathcal C_{p_j}}(\Psi_j\Phi_i(L_i),L_j).
\]

\begin{theorem}[exactness of mediated triangle transport]
\label{thm:intro-exactness}
Fix \((i,j)\).  The interaction profunctor
\[
\mathbb T_{ij}:
\mathcal C_{p_i}^{op}\times\mathcal C_{p_j}
\to
D(\mathrm{Vect})
\]
is exact in the source variable in the triangulated sense.  In particular, if
\[
X'\to X\to X''\to X'[1]
\]
is a distinguished triangle in \(\mathcal C_{p_i}\), then applying \(\Psi_j\Phi_i\) and then
\[
\RHom_{\mathcal C_{p_j}}(-,Y)
\]
produces a distinguished triangle in \(D(\mathrm{Vect})\), hence a long exact cohomology sequence for the interaction complexes \(\mathbb T_{ij}(X,Y)\).
\end{theorem}

\begin{definition}[triangle-visible localized content]
\label{def:intro-triangle-visible}
Let \(X\in\mathcal C_{p_j}\).  We say that \(X\) is \emph{right triangle-visible} with respect to \(L_j\) if there exists a distinguished triangle
\[
B\to X\xrightarrow{u}L_j\to B[1]
\]
with \(u\neq 0\).  We say that \(X\) is \emph{left triangle-visible} with respect to \(L_j\) if there exists a distinguished triangle
\[
L_j\xrightarrow{v}X\to C\to L_j[1]
\]
with \(v\neq 0\).
\end{definition}

\begin{proposition}[triangle-visible nonvanishing]
\label{prop:intro-triangle-visible}
Let \(X\in\mathcal C_{p_j}\).  If \(X\) is right triangle-visible with respect to \(L_j\), then
\[
\RHom_{\mathcal C_{p_j}}(X,L_j)\not\simeq 0.
\]
If \(X\) is left triangle-visible with respect to \(L_j\), then
\[
\RHom_{\mathcal C_{p_j}}(L_j,X)\not\simeq 0.
\]
In particular, if
\[
A_{ij}:=\Psi_j\Phi_i(L_i)
\]
is right triangle-visible with respect to \(L_j\), then
\[
\mathsf H_{ij}
=
\RHom_{\mathcal C_{p_j}}(A_{ij},L_j)
\not\simeq 0.
\]
\end{proposition}

\begin{theorem}[conditional bridge theorem]
\label{thm:intro-bridge}
Assume the probe, content, and detector hypotheses formulated in Section~\ref{sec:probe-content-detector}.  If an ordered pair \((i,j)\) is supported, meaning
\[
I_\Sigma^{(0/1)}(i,j)=1,
\]
then the associated pairwise interaction complex
\[
\mathsf H_{ij}
=
\RHom_{\mathcal C_{p_j}}(\Psi_j\Phi_i(L_i),L_j)
\]
is nonzero.
\end{theorem}

\begin{definition}[graded interaction polynomial]
\label{def:intro-graded}
Under the finiteness convention above, define
\[
P_{ij}(q)
:=
\sum_m \dim H^m(\mathsf H_{ij})q^m.
\]
The associated graded interaction matrix is
\[
I_\Sigma^{\mathrm{gr}}
:=
(P_{ij}(q))_{1\le i,j\le r}.
\]
\end{definition}

\begin{theorem}[graded pairwise interaction package]
\label{thm:intro-graded-package}
The matrix
\[
I_\Sigma^{\mathrm{gr}}
\]
is the first graded pairwise interaction package extracted from the MTT datum.  It refines the binary support law
\[
I_\Sigma^{(0/1)}
\]
of \cite{RahmanQuiverDataII}.  Under the hypotheses of Theorem~\ref{thm:intro-bridge}, every supported channel determines a nonzero graded interaction polynomial.
\end{theorem}

\begin{corollary}[inherited finite package]
\label{cor:intro-inherited}
The package
\[
(A_\Sigma,I_\Sigma^{(0/1)},I_\Sigma^{\mathrm{gr}})
\]
is the finite inherited algebraic output of the present paper.  It is the input passed to the later stability, BPS, and wall-crossing stage.
\end{corollary}

\subsection{Role of the paper and organization}

The paper is organized as a bridge from binary support to graded interaction.  Section~2 recalls the state-data package of \cite{RahmanQuiverDataI}, the support-level interaction package of \cite{RahmanQuiverDataII}, and the corrected-extension shadow package of \cite{RahmanPerverseNearbyCycles,RahmanMixedHodgeModules}.  Section~3 shows why the naive extension-only route remains local and does not by itself produce the desired off-diagonal interaction layer.  Section~4 defines MTT data, interaction profunctors, and probe evaluations.

Section~5 proves exactness and the associated long exact interaction sequences.  Section~6 introduces right and left triangle-visibility and proves the corresponding nonvanishing criterion.  Section~7 formulates the probe, content, and detector hypotheses.  Section~8 proves the conditional bridge theorem, upgrading supported channels to nontrivial interaction complexes.  Section~9 defines the graded polynomials \(P_{ij}(q)\) and the matrix \(I_\Sigma^{\mathrm{gr}}\).  Section~10 records the inherited package
\[
(A_\Sigma,I_\Sigma^{(0/1)},I_\Sigma^{\mathrm{gr}})
\]
as the finite algebraic output carried forward to the later stability and BPS sequel.

%----------------------------------------------------
\section{Finite-node input from previous work}
\label{sec:finite-node-input}

We recall the finite-node data used throughout the paper.  The inputs are:
\[
A_\Sigma=(V_\Sigma,E_\Sigma,c_\Sigma),
\qquad
I_\Sigma^{(0/1)},
\qquad
IC_{X_0}\to \mathcal P\to Q_\Sigma\to IC_{X_0}[1].
\]
The first is the state-data package of \cite{RahmanQuiverDataI}; the second is the support-level interaction package of \cite{RahmanQuiverDataII}; the third is the corrected-extension shadow package from \cite{RahmanPerverseNearbyCycles,RahmanMixedHodgeModules}.  This section fixes notation and records only the inherited structures needed below.

\begin{remark}[notation]
Throughout, \(\Q\) denotes the rational coefficient field, while \(Q_k=i_{k*}\Q_{\{p_k\}}\) denotes the point-supported shadow object at the node \(p_k\), and \(Q_\Sigma=\bigoplus_k Q_k\) denotes their finite direct sum.
\end{remark}

\subsection{State data}

Let
\[
\Sigma=\{p_1,\dots,p_r\}\subset X_0
\]
be the finite node set.  \cite{RahmanQuiverDataI} associates to \(\Sigma\) the finite state-data package
\[
A_\Sigma=(V_\Sigma,E_\Sigma,c_\Sigma)
\]
\cite{RahmanQuiverDataI}.  Here
\[
V_\Sigma=\{v_1,\dots,v_r\}
\]
is the vertex set indexed by the nodes, and
\[
E_\Sigma
=
\Ext^1_{\Perv(X_0;\Q)}(Q_\Sigma,IC_{X_0})
\]
is the extension space determined by the corrected finite-node quotient, where
\[
Q_\Sigma=\bigoplus_{k=1}^r Q_k,
\qquad
Q_k=i_{k*}\Q_{\{p_k\}}.
\]
The finite-node extension structure gives the nodewise decomposition
\[
E_\Sigma\cong \bigoplus_{k=1}^r \Q e_k,
\]
so that \(E_\Sigma\) carries the distinguished nodewise basis
\[
\{e_1,\dots,e_r\}
\]
\cite{RahmanQuiverDataI}.

The coefficient vector
\[
c_\Sigma=(c_1,\dots,c_r)\in\Q^r
\]
records the corrected global class in this basis:
\[
[\mathcal P]_{\mathrm{perv}}
=
\sum_{k=1}^r c_k e_k.
\]
Thus \(A_\Sigma\) is the finite algebraic state package inherited from \cite{RahmanQuiverDataI}.  In the present paper, \(A_\Sigma\) is fixed input rather than a structure to be reconstructed.

\subsection{Support-level interaction}

The second input is the support-level interaction package of  \cite{RahmanQuiverDataII}.  The categorical source of that package is the finite-node schober datum: a bulk category
\[
\mathcal C_{\mathrm{bulk}},
\]
localized categories
\[
\mathcal C_{p_i},
\qquad 1\le i\le r,
\]
and exact attachment functors
\[
\Phi_i:\mathcal C_{p_i}\to \mathcal C_{\mathrm{bulk}},
\qquad
\Psi_i:\mathcal C_{\mathrm{bulk}}\to \mathcal C_{p_i}.
\]
We write
\[
F_\Sigma=\{(\Phi_i,\Psi_i)\}_{i=1}^r
\]
for this finite transport package.

\cite{RahmanQuiverDataII} uses this transport data to construct a binary support law
\[
I_\Sigma^{(0/1)}.
\]
Equivalently, \(I_\Sigma^{(0/1)}\) records the finite support relation on the relevant vertex set: it says where a mediated channel is present, but not how the channel is graded or how it behaves under exact triangles \cite{RahmanQuiverDataII}.  This distinction is the point of departure for the present paper.  The input from \cite{RahmanQuiverDataII} is the support skeleton of pairwise interaction; the output constructed here is a graded interaction refinement.

To avoid notation conflict, we reserve \(Q_\Sigma\) for the corrected shadow-side quotient
\[
Q_\Sigma=\bigoplus_{k=1}^r Q_k.
\]
The quiver-theoretic package inherited from \cite{RahmanQuiverDataII} will therefore be called the \emph{finite support package}, and its binary interaction matrix will be denoted throughout by
\[
I_\Sigma^{(0/1)}.
\]

\subsection{Corrected-extension shadow package}

The third input is the corrected-extension shadow package.  On the perverse-sheaf side, the corrected finite-node architecture is encoded by the distinguished triangle
\[
IC_{X_0}
\longrightarrow
\mathcal P
\longrightarrow
Q_\Sigma
\longrightarrow
IC_{X_0}[1],
\]
where
\[
Q_\Sigma=\bigoplus_{k=1}^r Q_k,
\qquad
Q_k=i_{k*}\Q_{\{p_k\}}.
\]
This triangle is the shadow-side extension architecture underlying the state-data package \(A_\Sigma\) \cite{RahmanPerverseNearbyCycles,RahmanQuiverDataI}.  Its mixed-Hodge-module refinement supplies the corresponding Hodge-theoretic version of the same finite-node quotient structure \cite{RahmanMixedHodgeModules}.

The point-supported objects \(Q_k\) are the shadow-side localized sectors.  In the schober language, they correspond to the localized categories \(\mathcal C_{p_k}\) and to the probe objects introduced later.  Thus the corrected-extension shadow package provides the compatibility target for the categorical transport formalism: the MTT construction is required to refine the support law while remaining anchored to this corrected finite-node extension structure.

In the sequel, the phrase \emph{corrected-extension shadow package} refers to the finite-node shadow data
\[
IC_{X_0}
\longrightarrow
\mathcal P
\longrightarrow
Q_\Sigma
\longrightarrow
IC_{X_0}[1],
\qquad
Q_\Sigma=\bigoplus_{k=1}^r Q_k.
\]
The later sections construct a categorical pairwise interaction formalism compatible with this package and refining the binary support matrix \(I_\Sigma^{(0/1)}\).

%----------------------------------------------------
\section{Obstruction to the naive extension-only bridge}
\label{sec:extension-only-obstruction}

Before introducing mediated triangle transport, we isolate the limitation of the most direct shadow-side construction.  The corrected-extension package already gives the finite extension space
\[
E_\Sigma
=
\Ext^1_{\Perv(X_0;\Q)}(Q_\Sigma,IC_{X_0})
\cong
\bigoplus_{k=1}^r \Q e_k
\]
and its nodewise basis \(\{e_1,\dots,e_r\}\) \cite{RahmanQuiverDataI}.  It is therefore natural to ask whether the first pairwise interaction layer can be constructed using only the point-supported objects
\[
Q_i=i_{i*}\Q_{\{p_i\}}
\]
and their extension data.  The answer is negative in the sense needed here: this route records local coupling to the corrected extension, but it does not by itself produce the bulk-mediated off-diagonal transport mechanism required for the graded interaction package.

\subsection{The local extension-valued attempt}

The corrected-extension shadow package is encoded by the triangle
\[
IC_{X_0}
\longrightarrow
\mathcal P
\longrightarrow
Q_\Sigma
\longrightarrow
IC_{X_0}[1],
\qquad
Q_\Sigma=\bigoplus_{k=1}^r Q_k,
\qquad
Q_k=i_{k*}\Q_{\{p_k\}}.
\]
Since \(Q_\Sigma\) is a finite direct sum of point-supported localized objects, the corrected extension class decomposes into nodewise components.  Equivalently, one obtains the finite extension space
\[
E_\Sigma
=
\Ext^1_{\Perv(X_0;\Q)}(Q_\Sigma,IC_{X_0})
\cong
\bigoplus_{k=1}^r \Q e_k
\]
from \cite{RahmanQuiverDataI}.  Restricting the corrected global extension class to the \(i\)-th summand gives a nodewise class
\[
\delta_i
\in
\Ext^1_{\Perv(X_0;\Q)}(Q_i,IC_{X_0}),
\qquad
1\le i\le r.
\]

This suggests a first extension-valued approach to pairwise interaction: one might seek secondary classes
\[
\beta_{ij}
\in
\Ext^2_{\Perv(X_0;\Q)}(Q_i,Q_j),
\]
or related extension-theoretic objects, as coefficients measuring interaction between localized shadow sectors.  Such an approach is natural because it remains entirely within the corrected-extension shadow package and does not introduce additional categorical transport data.

However, this is not the interaction layer needed in the present sequence.  The desired object should measure how data from the \(i\)-th localized sector are transported through the common bulk and detected in the \(j\)-th localized sector.  Extension classes built only from the point-supported summands \(Q_i\) and \(Q_j\) do not, by themselves, supply that transport mechanism.

\subsection{Locality of the extension-only construction}

The obstruction is structural.  The objects \(Q_i\) are localized shadow terms; they isolate nodewise correction data in the quotient
\[
Q_\Sigma=\bigoplus_{k=1}^r Q_k.
\]
The classes
\[
\delta_i\in\Ext^1(Q_i,IC_{X_0})
\]
therefore record how the \(i\)-th shadow sector couples to the common object \(IC_{X_0}\).  They do not define an exact functor carrying data from the \(i\)-th localized sector to the \(j\)-th localized sector.

Thus, even if one can formally write extension-valued objects such as
\[
\beta_{ij}\in\Ext^2(Q_i,Q_j),
\]
the corrected-extension shadow package alone does not identify them as the primary off-diagonal pairwise interaction law.  Their natural role is extension-theoretic and shadow-side.  By contrast, the present paper requires an interaction object that is transport-theoretic, directed, and compatible with distinguished triangles.

This distinction is precisely where the schober input becomes necessary.  The composites
\[
\Psi_j\Phi_i
\]
are not merely shadow-side extension classes; they are exact bulk-mediated transport functors between localized categories.  They therefore retain information that the extension-only construction does not make active.

\subsection{Obstruction proposition}

\begin{proposition}[obstruction to the naive extension-only bridge]
\label{prop:obstruction-naive-extension-only-bridge}
The shadow-side extension-only construction determined by the localized objects \(Q_i\) and the nodewise classes
\[
\delta_i\in\Ext^1(Q_i,IC_{X_0})
\]
is intrinsically local.  It records nodewise coupling to the corrected extension, but it does not by itself furnish the bulk-mediated off-diagonal interaction layer required for the later stability and BPS formalism.
\end{proposition}

\begin{proof}
The corrected-extension shadow package consists of the finite quotient
\[
Q_\Sigma=\bigoplus_{k=1}^r Q_k,
\qquad
Q_k=i_{k*}\Q_{\{p_k\}},
\]
together with the extension triangle
\[
IC_{X_0}
\longrightarrow
\mathcal P
\longrightarrow
Q_\Sigma
\longrightarrow
IC_{X_0}[1].
\]
Its associated extension space decomposes nodewise:
\[
E_\Sigma
=
\Ext^1(Q_\Sigma,IC_{X_0})
\cong
\bigoplus_{k=1}^r \Q e_k
\]
\cite{RahmanQuiverDataI}.  Hence the resulting classes
\[
\delta_i\in\Ext^1(Q_i,IC_{X_0})
\]
are obtained by restricting the corrected extension to individual point-supported sectors.

This construction is local in the following sense: its basic objects are the point-supported summands \(Q_i\), and its basic classes describe how each summand couples to \(IC_{X_0}\).  No exact process is produced that transports an object of the \(i\)-th localized sector into the \(j\)-th localized sector while preserving triangulated structure.  Therefore the extension-only package does not supply the derived, directed, bulk-mediated object needed for pairwise interaction.

The required transport-theoretic mechanism is instead provided by the finite-node schober datum.  Its exact functors
\[
\Phi_i:\mathcal C_{p_i}\to\mathcal C_{\mathrm{bulk}},
\qquad
\Psi_j:\mathcal C_{\mathrm{bulk}}\to\mathcal C_{p_j}
\]
give composites
\[
\Psi_j\Phi_i:\mathcal C_{p_i}\to\mathcal C_{p_j},
\]
which are organized at the support level in \cite{RahmanQuiverDataII} and arise from the bulk/localized categorical package of \cite{RahmanSchoberPaper,RahmanMultiNodeSchoberPaper}.  These composites, unlike the nodewise extension classes alone, provide the exact transport structure used in the next section.  Thus the extension-only route is insufficient for the bulk-mediated off-diagonal interaction layer sought here.
\end{proof}

\subsection{Need for a transport-based bridge}

Proposition~\ref{prop:obstruction-naive-extension-only-bridge} motivates replacing the local extension-only route by a transport-based pairwise bridge.  For every ordered pair \((i,j)\), the finite-node schober package supplies the exact composite
\[
\Psi_j\Phi_i:
\mathcal C_{p_i}
\longrightarrow
\mathcal C_{p_j}.
\]
This functor keeps the bulk sector as an active intermediary and carries triangulated structure from the source localized category to the target localized category.

The next section packages these composites, together with localized probes and shadow compatibility, into mediated triangle transport.  The corresponding interaction objects are built from
\[
\RHom_{\mathcal C_{p_j}}(\Psi_j\Phi_i(-),-),
\]
rather than from the point-supported shadow objects alone.
%----------------------------------------------------
\section{Mediated triangle transport data}
\label{sec:mtt-data}

By Proposition~\ref{prop:obstruction-naive-extension-only-bridge}, the extension-only shadow construction does not supply the bulk-mediated off-diagonal interaction layer needed here.  We therefore use the transport structure already present in the finite-node schober package.  For each ordered pair \((i,j)\), the composite
\[
\Psi_j\Phi_i:
\mathcal C_{p_i}\to\mathcal C_{p_j}
\]
carries localized data from the \(i\)-th node sector through the common bulk and back into the \(j\)-th node sector.  Mediated triangle transport packages these composites, localized probes, and shadow compatibility into the datum from which the pairwise interaction complexes are extracted.

\subsection{Bulk and localized triangulated sectors}

Let
\[
\Sigma=\{p_1,\dots,p_r\}
\]
be the finite node set.  Following the finite-node schober construction of
\cite{RahmanSchoberPaper,RahmanMultiNodeSchoberPaper}, we fix a bulk triangulated category
\[
\mathcal C_{\mathrm{bulk}}
\]
and localized triangulated categories
\[
\mathcal C_{p_i},
\qquad
1\le i\le r.
\]
The bulk category represents the common global sector, while \(\mathcal C_{p_i}\) represents the categorical sector attached to \(p_i\).  These localized categories refine the same node sectors whose shadow-side objects are
\[
Q_i=i_{i*}\Q_{\{p_i\}}
\]
in the corrected-extension package of \cite{RahmanPerverseNearbyCycles,RahmanMixedHodgeModules}.

\subsection{Exact transport functors}

For each node \(p_i\), we fix exact functors
\[
\Phi_i:\mathcal C_{p_i}\to \mathcal C_{\mathrm{bulk}},
\qquad
\Psi_i:\mathcal C_{\mathrm{bulk}}\to \mathcal C_{p_i}.
\]
These are the attachment functors of the finite-node schober datum \cite{RahmanSchoberPaper,RahmanMultiNodeSchoberPaper}.  The functor \(\Phi_i\) transports data from the \(i\)-th localized sector into the bulk, while \(\Psi_i\) transports data from the bulk back to the \(i\)-th localized sector.

For each ordered pair \((i,j)\), define
\[
A_{ij}:=\Psi_j\Phi_i:
\mathcal C_{p_i}\to \mathcal C_{p_j}.
\]
For \(X\in\mathcal C_{p_i}\), we write
\[
A_{ij}(X)=\Psi_j\Phi_i(X)\in\mathcal C_{p_j}.
\]
This exact composite is the basic bulk-mediated transport mechanism used throughout the paper.

\subsection{Localized probes}

For each \(p_i\in\Sigma\), we include in the MTT datum a distinguished object
\[
L_i\in\mathcal C_{p_i},
\]
called the \emph{\(i\)-th localized probe}.  The probes are not asserted here to be uniquely constructed in maximal generality.  Rather, they are part of the chosen MTT datum and are required to be compatible with the corrected-extension shadow package.  Concretely, \(L_i\) is intended to represent, on the categorical side, the shadow-side point-supported object
\[
Q_i=i_{i*}\Q_{\{p_i\}}.
\]
The precise probe, content, and detector assumptions needed for the bridge theorem are formulated in Section~\ref{sec:probe-content-detector}.

\subsection{Corrected-extension shadow compatibility}

The MTT datum is required to remain anchored to the corrected-extension shadow package
\[
IC_{X_0}
\longrightarrow
\mathcal P
\longrightarrow
Q_\Sigma
\longrightarrow
IC_{X_0}[1],
\qquad
Q_\Sigma=\bigoplus_{k=1}^r Q_k,
\qquad
Q_k=i_{k*}\Q_{\{p_k\}},
\]
constructed on the perverse and mixed-Hodge-module sides in
\cite{RahmanPerverseNearbyCycles,RahmanMixedHodgeModules}.  We encode this compatibility by a shadow mechanism
\[
\operatorname{Sh}
\]
relating the categorical bulk/localized sectors to the corrected-extension package. For the purposes of the present paper, we regard
\[
\operatorname{Sh}:\mathcal C_{p_j}\longrightarrow D^b(\Perv(X_0;\Q))
\]
or, in the Hodge-theoretic refinement, the corresponding mixed-Hodge-module realization, as a shadow mechanism landing in a category where derived Hom objects with the point-supported sectors \(Q_j\) are defined.  We do not construct \(\operatorname{Sh}\) here; it is part of the MTT compatibility datum inherited from the finite-node schober shadow of \cite{RahmanSchoberPaper,RahmanMultiNodeSchoberPaper}.

At minimum, \(\operatorname{Sh}\) identifies each probe \(L_i\) with the corresponding localized shadow sector \(Q_i\) in the sense required by the later hypotheses.  Thus the probes are categorical avatars of the point-supported summands of \(Q_\Sigma\), and the transport formalism is measured against the corrected finite-node extension structure rather than introduced independently.

\subsection{Derived interaction profunctors}

We now define the basic pairwise interaction objects.

\begin{definition}[MTT interaction profunctor]
\label{def:mtt-interaction-profunctor}
For each ordered pair \((i,j)\), the \emph{MTT interaction profunctor} is
\[
\mathbb T_{ij}:
\mathcal C_{p_i}^{op}\times\mathcal C_{p_j}
\to
D(\mathrm{Vect})
\]
defined by
\[
\mathbb T_{ij}(X,Y)
:=
\RHom_{\mathcal C_{p_j}}(\Psi_j\Phi_i(X),Y)
=
\RHom_{\mathcal C_{p_j}}(A_{ij}(X),Y).
\]
\end{definition}

The notation reflects the fact that \(\mathbb T_{ij}\) is contravariant in \(X\in\mathcal C_{p_i}\) and covariant in \(Y\in\mathcal C_{p_j}\), hence is a derived interaction profunctor rather than an ordinary functor between localized categories \cite{BenabouProfunctors, LoregianCoendCalculus}.  Its role is to assign a derived pairwise interaction object to each source--target pair of localized objects.

\begin{convention}[linear and finiteness conventions]
\label{conv:mtt-finiteness}
We work in a \(k\)-linear triangulated setting in which the derived Hom objects appearing below are defined in \(D(\mathrm{Vect})\).  When graded decategorifications are formed in Section~\ref{sec:graded-interaction-package}, we assume the probe-evaluation complexes
\[
\RHom_{\mathcal C_{p_j}}(\Psi_j\Phi_i(L_i),L_j)
\]
have finite-dimensional bounded cohomology, or else locally finite cohomology so that the resulting graded expression is interpreted as a formal Poincare series.
\end{convention}

\subsection{Probe evaluation}

Evaluating the interaction profunctor on localized probes gives the pairwise interaction complex.

\begin{definition}[pairwise interaction complex]
\label{def:pairwise-interaction-complex}
For each ordered pair \((i,j)\), define
\[
\mathsf H_{ij}
:=
\mathbb T_{ij}(L_i,L_j)
=
\RHom_{\mathcal C_{p_j}}(\Psi_j\Phi_i(L_i),L_j)
=
\RHom_{\mathcal C_{p_j}}(A_{ij}(L_i),L_j).
\]
\end{definition}

The complexes \(\mathsf H_{ij}\) are the derived pairwise interaction objects extracted from the MTT datum.  They refine the binary support entries of \(I_\Sigma^{(0/1)}\) by replacing support-only information with derived Hom data:
\[
I_\Sigma^{(0/1)}
\quad\leadsto\quad
\mathsf H_{ij}.
\]
The next section proves the exactness properties that make these complexes useful for later graded interaction and BPS-facing constructions.

%----------------------------------------------------
\section{Exactness and triangle transport}
\label{sec:exactness-triangle-transport}

Mediated triangle transport is useful because it preserves exact-triangle structure.  Once interaction is encoded by
\[
\mathbb T_{ij}(X,Y)
=
\RHom_{\mathcal C_{p_j}}(A_{ij}(X),Y),
\qquad
A_{ij}=\Psi_j\Phi_i,
\]
distinguished triangles in the source localized category give distinguished triangles, and hence long exact cohomology sequences, in the derived interaction target.  This section records the formal exactness properties used later.

\subsection{Exactness of bulk-mediated transport}

For each ordered pair \((i,j)\), recall that
\[
A_{ij}
=
\Psi_j\Phi_i:
\mathcal C_{p_i}\to \mathcal C_{p_j}.
\]
The functors \(\Phi_i\) and \(\Psi_j\) are exact by the finite-node schober input \cite{RahmanSchoberPaper,RahmanMultiNodeSchoberPaper}.  Hence their composite is exact in the usual triangulated sense \cite{KS,DimcaSheaves,CDW}.

\begin{proposition}[exactness of bulk-mediated transport]
\label{prop:exactness-bulk-mediated-transport}
For each ordered pair \((i,j)\), the composite
\[
A_{ij}
=
\Psi_j\Phi_i:
\mathcal C_{p_i}\to \mathcal C_{p_j}
\]
is an exact functor of triangulated categories.
\end{proposition}

\begin{proof}
Both \(\Phi_i\) and \(\Psi_j\) are exact.  The composite of exact functors between triangulated categories is exact, so \(A_{ij}=\Psi_j\Phi_i\) is exact.
\end{proof}

\begin{corollary}[transport of distinguished triangles]
\label{cor:transport-distinguished-triangles}
Let
\[
X'\to X\to X''\to X'[1]
\]
be a distinguished triangle in \(\mathcal C_{p_i}\).  Then
\[
A_{ij}(X')\to A_{ij}(X)\to A_{ij}(X'')\to A_{ij}(X')[1]
\]
is a distinguished triangle in \(\mathcal C_{p_j}\).
\end{corollary}

\begin{proof}
This follows immediately from Proposition~\ref{prop:exactness-bulk-mediated-transport}.
\end{proof}

Thus localized data can be transported through the bulk without losing the distinguished-triangle structure needed for the interaction theory.

\subsection{Exactness in the source variable}

Fix \((i,j)\) and \(Y\in\mathcal C_{p_j}\).  The source-variable part of the interaction profunctor is
\[
X
\longmapsto
\mathbb T_{ij}(X,Y)
=
\RHom_{\mathcal C_{p_j}}(A_{ij}(X),Y).
\]
Because \(X\) appears in the first argument of \(\RHom\), the resulting exactness is contravariant: a distinguished triangle in \(\mathcal C_{p_i}\) produces the reversed distinguished triangle in \(D(\mathrm{Vect})\).

\begin{proposition}[MTT exactness in the source variable]
\label{prop:mtt-exactness-source-variable}
Fix \((i,j)\) and \(Y\in\mathcal C_{p_j}\).  If
\[
X'\to X\to X''\to X'[1]
\]
is a distinguished triangle in \(\mathcal C_{p_i}\), then
\[
\mathbb T_{ij}(X'',Y)
\to
\mathbb T_{ij}(X,Y)
\to
\mathbb T_{ij}(X',Y)
\to
\mathbb T_{ij}(X'',Y)[1]
\]
is a distinguished triangle in \(D(\mathrm{Vect})\).
\end{proposition}

\begin{proof}
By Corollary~\ref{cor:transport-distinguished-triangles},
\[
A_{ij}(X')\to A_{ij}(X)\to A_{ij}(X'')\to A_{ij}(X')[1]
\]
is a distinguished triangle in \(\mathcal C_{p_j}\).  Applying
\[
\RHom_{\mathcal C_{p_j}}(-,Y)
\]
gives a distinguished triangle in \(D(\mathrm{Vect})\), with the order reversed because \(\RHom(-,Y)\) is contravariant in its first argument.  Substituting
\[
\mathbb T_{ij}(Z,Y)
=
\RHom_{\mathcal C_{p_j}}(A_{ij}(Z),Y)
\]
for \(Z=X',X,X''\) gives the asserted triangle. The connecting morphism is the standard one obtained by applying the contravariant derived Hom functor to the rotated distinguished triangle; its sign depends only on the usual triangulated-category convention and plays no role in the cohomological long exact (interaction) sequence below.
\end{proof}

\begin{remark}
\label{rem:DVect-remark}
Throughout the paper, \(D(\mathrm{Vect})\) denotes the derived category of complexes of vector spaces.  Distinguished triangles in \(D(\mathrm{Vect})\) yield the usual long exact cohomology sequences \cite{DimcaSheaves,KS}.
\end{remark}

\subsection{Long exact interaction sequences}

The preceding exact triangle gives the cohomological form of mediated triangle transport.

\begin{theorem}[long exact interaction sequence]
\label{thm:long-exact-interaction-sequence}
Fix \((i,j)\) and \(Y\in\mathcal C_{p_j}\).  If
\[
X'\to X\to X''\to X'[1]
\]
is a distinguished triangle in \(\mathcal C_{p_i}\), then there is a long exact sequence
\[
\cdots
\to
H^m(\mathbb T_{ij}(X'',Y))
\to
H^m(\mathbb T_{ij}(X,Y))
\to
H^m(\mathbb T_{ij}(X',Y))
\to
H^{m+1}(\mathbb T_{ij}(X'',Y))
\to
\cdots .
\]
\end{theorem}

\begin{proof}
By Proposition~\ref{prop:mtt-exactness-source-variable},
\[
\mathbb T_{ij}(X'',Y)
\to
\mathbb T_{ij}(X,Y)
\to
\mathbb T_{ij}(X',Y)
\to
\mathbb T_{ij}(X'',Y)[1]
\]
is a distinguished triangle in \(D(\mathrm{Vect})\).  Taking cohomology gives the associated long exact sequence.
\end{proof}

Theorem~\ref{thm:long-exact-interaction-sequence} is the formal reason that \(\mathsf H_{ij}\)-type objects carry more information than the support entries of \(I_\Sigma^{(0/1)}\).  They behave functorially with respect to distinguished triangles in the source localized sector and therefore admit cohomological exactness constraints.

\begin{remark}[target-variable exactness]
\label{rem:target-variable-exactness}
For fixed \(X\in\mathcal C_{p_i}\), one may also consider
\[
Y
\longmapsto
\mathbb T_{ij}(X,Y)
=
\RHom_{\mathcal C_{p_j}}(A_{ij}(X),Y).
\]
In settings where \(\RHom_{\mathcal C_{p_j}}(A_{ij}(X),-)\) is exact in the second variable, distinguished triangles in the target variable give corresponding distinguished triangles and long exact cohomology sequences.  The present paper does not require this additional target-variable exactness, so it is not included among the standing hypotheses.
\end{remark}

%----------------------------------------------------
\section{Triangle-visible localized content}
\label{sec:triangle-visible-content}

This section isolates a simple nonvanishing mechanism for the interaction complexes.  If the target probe \(L_j\) is visible in the transported object through a distinguished triangle with a nonzero visibility morphism, then the corresponding derived Hom complex cannot vanish.

\subsection{Right and left triangle visibility}

Let \(j\) be fixed, let
\[
L_j\in\mathcal C_{p_j}
\]
be the distinguished localized probe, and let
\[
X\in\mathcal C_{p_j}.
\]

\begin{definition}[right and left triangle visibility]
\label{def:triangle-visibility}
We say that \(X\) is \emph{right triangle-visible} with respect to \(L_j\) if there exists a distinguished triangle
\[
B\to X\xrightarrow{u}L_j\to B[1]
\]
with \(u\neq 0\).  We say that \(X\) is \emph{left triangle-visible} with respect to \(L_j\) if there exists a distinguished triangle
\[
L_j\xrightarrow{v}X\to C\to L_j[1]
\]
with \(v\neq 0\).
\end{definition}

Right visibility detects a nonzero morphism from \(X\) to the probe.  Left visibility detects a nonzero morphism from the probe to \(X\).  Neither condition requires \(L_j\) to occur as a direct summand of \(X\); it only requires the probe to appear through a distinguished-triangle morphism.  This is weaker than a decomposition statement but stronger than support-level incidence alone.

\subsection{Triangle-visible nonvanishing}

\begin{proposition}[triangle-visible nonvanishing]
\label{prop:triangle-visible-nonvanishing}
Let \(X\in\mathcal C_{p_j}\).  If \(X\) is right triangle-visible with respect to \(L_j\), then
\[
\RHom_{\mathcal C_{p_j}}(X,L_j)\not\simeq 0.
\]
If \(X\) is left triangle-visible with respect to \(L_j\), then
\[
\RHom_{\mathcal C_{p_j}}(L_j,X)\not\simeq 0.
\]
\end{proposition}

\begin{proof}
If \(X\) is right triangle-visible, then there is a distinguished triangle
\[
B\to X\xrightarrow{u}L_j\to B[1]
\]
with \(u\neq 0\).  Hence
\[
\Hom_{\mathcal C_{p_j}}(X,L_j)\neq 0.
\]
Since
\[
\Hom_{\mathcal C_{p_j}}(X,L_j)
=
H^0\!\left(\RHom_{\mathcal C_{p_j}}(X,L_j)\right),
\]
the object \(\RHom_{\mathcal C_{p_j}}(X,L_j)\) is not quasi-isomorphic to zero.

The left-visible case is identical.  If
\[
L_j\xrightarrow{v}X\to C\to L_j[1]
\]
is a distinguished triangle with \(v\neq 0\), then
\[
\Hom_{\mathcal C_{p_j}}(L_j,X)\neq 0,
\]
so
\[
H^0\!\left(\RHom_{\mathcal C_{p_j}}(L_j,X)\right)\neq 0.
\]
Thus \(\RHom_{\mathcal C_{p_j}}(L_j,X)\not\simeq 0\).
\end{proof}

\begin{remark}[directionality]
\label{rem:triangle-visible-directionality}
The two visibility conditions have different variance.  Right visibility gives nonvanishing of
\[
\RHom_{\mathcal C_{p_j}}(X,L_j),
\]
whereas left visibility gives nonvanishing of
\[
\RHom_{\mathcal C_{p_j}}(L_j,X).
\]
The primary pairwise interaction complexes of this paper have the first form, so right visibility is the condition used below.
\end{remark}

\subsection{Mediated triangle visibility}

We now apply this criterion to transported probes.  For an ordered pair \((i,j)\), recall that
\[
A_{ij}(L_i)
=
\Psi_j\Phi_i(L_i)
\in
\mathcal C_{p_j}.
\]

\begin{definition}[mediated right triangle visibility]
\label{def:mediated-triangle-visibility}
The mediated transport from \(L_i\) to the \(j\)-th localized sector is \emph{right triangle-visible with respect to \(L_j\)} if there exists a distinguished triangle
\[
B\to A_{ij}(L_i)\xrightarrow{u}L_j\to B[1]
\]
with \(u\neq 0\).
\end{definition}

This condition says that the target probe \(L_j\) is visible through a nonzero morphism out of the transported source probe
\[
A_{ij}(L_i)=\Psi_j\Phi_i(L_i).
\]
Since
\[
\mathsf H_{ij}
=
\RHom_{\mathcal C_{p_j}}(A_{ij}(L_i),L_j),
\]
right visibility is exactly the triangle-visible condition adapted to the variance of the pairwise interaction complex.

\begin{corollary}[mediated triangle-visible nonvanishing]
\label{cor:mediated-triangle-visible-nonvanishing}
If \(A_{ij}(L_i)=\Psi_j\Phi_i(L_i)\) is right triangle-visible with respect to \(L_j\), then
\[
\mathsf H_{ij}
=
\RHom_{\mathcal C_{p_j}}(\Psi_j\Phi_i(L_i),L_j)
\not\simeq 0.
\]
\end{corollary}

\begin{proof}
Apply Proposition~\ref{prop:triangle-visible-nonvanishing} with
\[
X=A_{ij}(L_i)=\Psi_j\Phi_i(L_i).
\]
\end{proof}

Corollary~\ref{cor:mediated-triangle-visible-nonvanishing} is a sufficient criterion for nonvanishing of the pairwise interaction complex.  The next section formulates weaker shadow-side hypotheses that imply the nonvanishing needed for the conditional bridge theorem.

%----------------------------------------------------
\section{Probe, content, and detector hypotheses}
\label{sec:probe-content-detector}

The previous sections established the formal transport mechanism and a triangle-visible nonvanishing criterion.  To upgrade support-level interaction to nontrivial pairwise interaction complexes, we now state the three hypotheses used in the conditional bridge theorem.  They are not stability or BPS assumptions; they are compatibility assumptions linking the support matrix, the corrected-extension shadow package, and the localized probes.

\subsection{Probe object hypothesis}

The first hypothesis fixes the localized probes used to evaluate the interaction profunctors.

\begin{hypothesis}[probe object hypothesis]
\label{hyp:probe-object}
For each node \(p_i\in\Sigma\), the localized category \(\mathcal C_{p_i}\) is equipped with a distinguished object
\[
L_i\in\mathcal C_{p_i},
\]
called the \emph{\(i\)-th localized probe}, such that:
\begin{enumerate}
\item \(L_i\) is specified as part of the MTT datum attached to \(p_i\);
\item under the shadow mechanism \(\operatorname{Sh}\), the probe \(L_i\) is compatible with the localized shadow object
\[
Q_i=i_{i*}\Q_{\{p_i\}};
\]
\item equivalences of MTT data preserving the shadow mechanism identify the corresponding probes.
\end{enumerate}
\end{hypothesis}

This hypothesis is the structure needed to form the probe evaluations
\[
\mathsf H_{ij}
=
\mathbb T_{ij}(L_i,L_j).
\]
It does not assert a general existence-and-uniqueness theorem for probes; it records the compatibility required once an MTT datum has been fixed.

\subsection{Weak content hypothesis}

The second hypothesis connects the binary support law of \cite{RahmanQuiverDataII} to shadow-side localized content.  For each ordered pair \((i,j)\), recall that
\[
A_{ij}(L_i)
=
\Psi_j\Phi_i(L_i)
\in
\mathcal C_{p_j}.
\]

\begin{hypothesis}[weak content hypothesis]
\label{hyp:weak-content}
Let \((i,j)\) satisfy
\[
I_\Sigma^{(0/1)}(i,j)=1.
\]
Then the transported probe \(A_{ij}(L_i)\) has nontrivial \(j\)-localized content on the shadow side: its shadow is not orthogonal to \(Q_j\).  Equivalently, at least one of the following derived Hom objects is nonzero:
\[
\RHom(\operatorname{Sh}(A_{ij}(L_i)),Q_j)\not\simeq 0,
\qquad
\RHom(Q_j,\operatorname{Sh}(A_{ij}(L_i)))\not\simeq 0.
\]
\end{hypothesis}

This condition does not require \(A_{ij}(L_i)\) to contain \(Q_j\) as a direct summand, quotient, or explicitly identified constituent.  It only requires that a supported channel remain visible to the \(j\)-th localized shadow sector after applying the shadow mechanism. 

The nonorthogonality condition is deliberately two-sided: either
\[
\RHom(\operatorname{Sh}(A_{ij}(L_i)),Q_j)\not\simeq 0
\]
or
\[
\RHom(Q_j,\operatorname{Sh}(A_{ij}(L_i)))\not\simeq 0
\]
is taken as evidence of \(j\)-localized shadow content.  The detector hypothesis is one-sided because the primary interaction complex has the variance
\[
\mathsf H_{ij}
=
\RHom_{\mathcal C_{p_j}}(A_{ij}(L_i),L_j).
\]

\subsection{Weak detector hypothesis}

The third hypothesis says that the target probe detects the shadow-side localized content relevant to the variance of the primary interaction complex. Only a special case of Hypothesis~\ref{hyp:weak-detector} is used in the proof of Theorem~\ref{thm:conditional-bridge}, namely the case
\[
X=A_{ij}(L_i)=\Psi_j\Phi_i(L_i).
\]
Thus the hypothesis could be weakened to objects lying in the transported-probe locus; we keep the broader formulation to state the detector property independently of a fixed ordered pair.

\begin{hypothesis}[weak detector hypothesis]
\label{hyp:weak-detector}
Let \(X\in\mathcal C_{p_j}\).  If \(\operatorname{Sh}(X)\) is not orthogonal to \(Q_j\) in the sense of Hypothesis~\ref{hyp:weak-content}, then
\[
\RHom_{\mathcal C_{p_j}}(X,L_j)\not\simeq 0.
\]
\end{hypothesis}

The hypothesis is one-sided because the pairwise interaction complex is
\[
\mathsf H_{ij}
=
\RHom_{\mathcal C_{p_j}}(A_{ij}(L_i),L_j).
\]
Thus only detection by \(L_j\) in the second argument is needed for the bridge theorem.  No generation statement for the full localized category \(\mathcal C_{p_j}\) is assumed.

\subsection{Scope of the hypotheses}

The hypotheses above are intentionally limited.  Hypothesis~\ref{hyp:probe-object} fixes compatible probes in a chosen MTT datum.  Hypothesis~\ref{hyp:weak-content} says that a supported channel has nontrivial shadow-side \(j\)-content.  Hypothesis~\ref{hyp:weak-detector} says that the target probe \(L_j\) detects that content through the derived Hom relevant to \(\mathsf H_{ij}\).

Stronger assumptions are possible but unnecessary here.  One could require \(L_j\) to generate a thick subcategory, require \(A_{ij}(L_i)\) to contain \(L_j\) as a direct summand, or require \(\operatorname{Sh}(A_{ij}(L_i))\) to contain \(Q_j\) explicitly.  The bridge theorem below uses only the weaker nonorthogonality-and-detection conditions above.

Sections~3--6 were formal consequences of the corrected-extension, schober, and MTT setup.  The nonvanishing conclusion of Theorem~\ref{thm:conditional-bridge} is the first result that depends on Hypotheses~\ref{hyp:probe-object}, \ref{hyp:weak-content}, and \ref{hyp:weak-detector}.

%----------------------------------------------------
\section{Conditional bridge theorem}
\label{sec:conditional-bridge}

We now combine the support law, the MTT transport composite, and the probe/content/detector hypotheses.  The result is the first nonvanishing bridge from binary support to derived pairwise interaction.

\subsection{Statement}

For each ordered pair \((i,j)\), recall that
\[
A_{ij}(L_i)
=
\Psi_j\Phi_i(L_i)
\in
\mathcal C_{p_j}
\]
and
\[
\mathsf H_{ij}
=
\RHom_{\mathcal C_{p_j}}(A_{ij}(L_i),L_j)
=
\RHom_{\mathcal C_{p_j}}(\Psi_j\Phi_i(L_i),L_j).
\]
The binary matrix \(I_\Sigma^{(0/1)}\) from \cite{RahmanQuiverDataII} records whether the mediated channel \((i,j)\) is supported.  The bridge theorem says that, under the hypotheses of Section~\ref{sec:probe-content-detector}, supported channels yield nonzero interaction complexes.

\begin{theorem}[conditional bridge theorem]
\label{thm:conditional-bridge}
Assume Hypotheses~\ref{hyp:probe-object}, \ref{hyp:weak-content}, and \ref{hyp:weak-detector}.  If
\[
I_\Sigma^{(0/1)}(i,j)=1,
\]
then
\[
\mathsf H_{ij}
=
\RHom_{\mathcal C_{p_j}}(\Psi_j\Phi_i(L_i),L_j)
\not\simeq 0.
\]
\end{theorem}

\subsection{Proof}

\begin{proof}
Let \((i,j)\) satisfy
\[
I_\Sigma^{(0/1)}(i,j)=1.
\]
The MTT composite sends the source probe to
\[
A_{ij}(L_i)
=
\Psi_j\Phi_i(L_i)
\in
\mathcal C_{p_j}.
\]
By Hypothesis~\ref{hyp:weak-content}, the shadow of \(A_{ij}(L_i)\) is not orthogonal to the localized shadow object \(Q_j\).  Applying Hypothesis~\ref{hyp:weak-detector} to
\[
X=A_{ij}(L_i)
\]
gives
\[
\RHom_{\mathcal C_{p_j}}(A_{ij}(L_i),L_j)\not\simeq 0.
\]
By Definition~\ref{def:pairwise-interaction-complex}, this object is precisely
\[
\mathsf H_{ij}.
\]
Hence \(\mathsf H_{ij}\not\simeq 0\).
\end{proof}

\subsection{Interpretation}

The theorem is a bridge result, not a computation of \(\mathsf H_{ij}\).  It shows that the binary support entry
\[
I_\Sigma^{(0/1)}(i,j)=1
\]
becomes a nonzero derived interaction complex once the supported channel has nontrivial shadow content and the target probe detects that content.

Thus the output is stronger than support but weaker than a BPS invariant. \cite{RahmanQuiverDataI} supplies the finite state variables \(A_\Sigma\); \cite{RahmanQuiverDataII} supplies the support skeleton \(I_\Sigma^{(0/1)}\); the present theorem supplies the nonvanishing step needed before forming the graded interaction package
\[
I_\Sigma^{\mathrm{gr}}.
\]
The theorem remains conditional because the probe, content, and detector hypotheses are part of the chosen MTT framework.  Its role is to identify exactly where support-level interaction becomes derived pairwise interaction.

%----------------------------------------------------
\section{The graded interaction package}
\label{sec:graded-interaction-package}

The conditional bridge theorem produces nonzero derived interaction complexes
\[
\mathsf H_{ij}
=
\RHom_{\mathcal C_{p_j}}(\Psi_j\Phi_i(L_i),L_j)
\]
on supported channels, under Hypotheses~\ref{hyp:probe-object}, \ref{hyp:weak-content}, and \ref{hyp:weak-detector}.  We now pass from these derived objects to their first graded algebraic shadows.

\subsection{Graded interaction polynomials}

We use the finiteness convention of Convention~\ref{conv:mtt-finiteness}.  Thus the cohomology groups of the probe-evaluation complexes are finite-dimensional and bounded, or locally finite if the expression below is interpreted as a formal Poincare series.

\begin{definition}[graded interaction polynomial]
\label{def:graded-interaction-polynomial}
For each ordered pair \((i,j)\), define
\[
P_{ij}(q)
:=
\sum_m \dim H^m(\mathsf H_{ij})\,q^m.
\]
\end{definition}

In the bounded Hom-finite case, \(P_{ij}(q)\) is a Laurent polynomial.  In the locally finite case, it is a formal graded Poincare series.  In either interpretation, \(P_{ij}(q)\) records the degree-by-degree cohomological size of the pairwise interaction complex.

\subsection{The graded interaction matrix}

\begin{definition}[graded interaction matrix]
\label{def:graded-interaction-matrix}
The \emph{graded interaction matrix} associated to the MTT datum is
\[
I_\Sigma^{\mathrm{gr}}
:=
(P_{ij}(q))_{1\le i,j\le r}.
\]
\end{definition}

The binary matrix \(I_\Sigma^{(0/1)}\) of \cite{RahmanQuiverDataII} records whether a mediated channel is supported.  The matrix \(I_\Sigma^{\mathrm{gr}}\) records the graded cohomological shadow of the corresponding derived interaction complex.  Thus the passage
\[
I_\Sigma^{(0/1)}
\quad\leadsto\quad
I_\Sigma^{\mathrm{gr}}
\]
is the passage from support-level interaction to graded pairwise interaction.

\subsection{Scalar specializations}

The polynomial \(P_{ij}(q)\) is the primary graded object, but two scalar shadows will be useful later.

\begin{definition}[specialized interaction weights]
\label{def:specialized-interaction-weights}
In the bounded Hom-finite case, define
\[
w_{ij}^{\mathrm{tot}}
:=
P_{ij}(1),
\qquad
w_{ij}^{\chi}
:=
P_{ij}(-1).
\]
In the locally finite case, these scalar specializations are defined only when the corresponding sums converge.
\end{definition}

Thus
\[
w_{ij}^{\mathrm{tot}}
=
\sum_m \dim H^m(\mathsf H_{ij})
\]
is the total cohomological size, while
\[
w_{ij}^{\chi}
=
\sum_m (-1)^m\dim H^m(\mathsf H_{ij})
\]
is the Euler-type signed shadow.  These are decategorified summaries of \(P_{ij}(q)\), not replacements for it.

\begin{remark}
\label{rem:specialization-remark}
The specialization \(w_{ij}^{\mathrm{tot}}\) forgets degree, while \(w_{ij}^{\chi}\) retains only the alternating parity contribution.  The graded polynomial \(P_{ij}(q)\) remains the primary interaction shadow in this paper.
\end{remark}

\subsection{The graded interaction theorem}

\begin{theorem}[graded interaction theorem]
\label{thm:graded-interaction-theorem}
Let
\[
I_\Sigma^{\mathrm{gr}}
=
(P_{ij}(q))_{1\le i,j\le r}
\]
be the graded interaction matrix associated to the MTT datum.  Then:
\begin{enumerate}
\item \(I_\Sigma^{\mathrm{gr}}\) is the first graded pairwise interaction package extracted from the MTT formalism;
\item \(I_\Sigma^{\mathrm{gr}}\) refines the binary support package \(I_\Sigma^{(0/1)}\) of \cite{RahmanQuiverDataII};
\item under the hypotheses of Theorem~\ref{thm:conditional-bridge}, every supported ordered pair \((i,j)\) satisfies
\[
P_{ij}(q)\neq 0.
\]
\end{enumerate}
\end{theorem}

\begin{proof}
The first assertion follows from Definitions~\ref{def:pairwise-interaction-complex}, \ref{def:graded-interaction-polynomial}, and \ref{def:graded-interaction-matrix}: the complexes \(\mathsf H_{ij}\) are the probe evaluations of the MTT interaction profunctors, and \(P_{ij}(q)\) records their graded cohomology.

For the second assertion, \(I_\Sigma^{(0/1)}\) records only support of mediated channels.  The matrix \(I_\Sigma^{\mathrm{gr}}\) attaches to each ordered pair the graded cohomological shadow of the corresponding derived interaction complex, and hence refines the support-level package.

For the third assertion, assume the hypotheses of Theorem~\ref{thm:conditional-bridge} and suppose
\[
I_\Sigma^{(0/1)}(i,j)=1.
\]
Then Theorem~\ref{thm:conditional-bridge} gives
\[
\mathsf H_{ij}\not\simeq 0.
\]
In either case of Convention~\ref{conv:mtt-finiteness}, this implies that some cohomology group
\(H^m(\mathsf H_{ij})\) is nonzero.  Hence
\[
P_{ij}(q)
=
\sum_m \dim H^m(\mathsf H_{ij})\,q^m
\neq 0,
\]
as a Laurent polynomial in the bounded Hom-finite case, or as a formal Poincare series in the locally finite case.
\end{proof}

Theorem~\ref{thm:graded-interaction-theorem} converts the derived MTT formalism into a finite graded algebraic package.  It remains weaker than a stability condition, BPS spectrum, or wall-crossing automorphism, but it is richer than binary incidence: it records not only whether a channel exists, but also where its interaction cohomology occurs.  This graded package is the principal output carried into the next stage of the sequence.

%----------------------------------------------------
\section{Compatibility and inheritance}
\label{sec:compatibility-inheritance}

We now record how the graded package fits into the sequence developed in \cite{RahmanQuiverDataI,RahmanQuiverDataII}.  The present paper does not replace the earlier finite packages; it adds the first graded pairwise layer on top of them:
\[
A_\Sigma
\quad\leadsto\quad
I_\Sigma^{(0/1)}
\quad\leadsto\quad
I_\Sigma^{\mathrm{gr}}.
\]

\subsection{Compatibility with the state-data package}

The graded interaction matrix remains indexed by the same finite node set used in \cite{RahmanQuiverDataI}.  If
\[
A_\Sigma=(V_\Sigma,E_\Sigma,c_\Sigma),
\qquad
V_\Sigma=\{v_1,\dots,v_r\},
\]
then
\[
I_\Sigma^{\mathrm{gr}}
=
(P_{ij}(q))_{1\le i,j\le r}
\]
is indexed by ordered pairs of the same nodes.  Thus the graded interaction package is built on the finite skeleton underlying \(A_\Sigma\) \cite{RahmanQuiverDataI}.

The compatibility is not only notational.  The localized probes
\[
L_i\in\mathcal C_{p_i}
\]
are required, through the probe hypothesis and the shadow mechanism of the MTT datum, to represent the same localized sectors whose shadow-side objects are
\[
Q_i=i_{i*}\Q_{\{p_i\}}
\]
in the corrected-extension package \cite{RahmanPerverseNearbyCycles,RahmanMixedHodgeModules}.  Hence the MTT construction refines the finite-node geometry at the level of pairwise interaction while remaining anchored to the corrected-extension data from which \(A_\Sigma\) was extracted.

Section~\ref{sec:extension-only-obstruction} explains why the extension-only route remains local.  In this sense, \(A_\Sigma\) records intrinsic nodewise state data, while \(I_\Sigma^{\mathrm{gr}}\) records derived pairwise interaction data obtained from bulk-mediated transport.  The two packages are complementary layers of the same finite-node architecture.

\subsection{Compatibility with the support package}

The relationship with \cite{RahmanQuiverDataII} is direct.  The binary matrix
\[
I_\Sigma^{(0/1)}
\]
records whether a mediated channel is supported \cite{RahmanQuiverDataII}.  The graded matrix
\[
I_\Sigma^{\mathrm{gr}}
=
(P_{ij}(q))
\]
records the cohomological shadow of the corresponding derived interaction complex
\[
\mathsf H_{ij}
=
\RHom_{\mathcal C_{p_j}}(\Psi_j\Phi_i(L_i),L_j).
\]
Thus
\[
I_\Sigma^{(0/1)}
\quad\leadsto\quad
I_\Sigma^{\mathrm{gr}}
\]
is a refinement from support-level incidence to graded pairwise interaction.

Under the hypotheses of Theorem~\ref{thm:conditional-bridge}, this refinement is nontrivial on supported channels:
\[
I_\Sigma^{(0/1)}(i,j)=1
\quad\Longrightarrow\quad
P_{ij}(q)\neq 0.
\]
Without those hypotheses, \(I_\Sigma^{\mathrm{gr}}\) is still defined from the MTT datum under the finiteness convention, but the nonvanishing-on-support conclusion is not asserted.

\subsection{Inheritance proposition}

\begin{proposition}[inheritance proposition]
\label{prop:inheritance-package}
The package
\[
\bigl(A_\Sigma,\ I_\Sigma^{(0/1)},\ I_\Sigma^{\mathrm{gr}}\bigr)
\]
is the finite algebraic output of the sequence through the present paper.  More precisely:
\begin{enumerate}
\item \(A_\Sigma=(V_\Sigma,E_\Sigma,c_\Sigma)\) is the finite state-data package of \cite{RahmanQuiverDataI};
\item \(I_\Sigma^{(0/1)}\) is the support-level interaction package of \cite{RahmanQuiverDataII};
\item \(I_\Sigma^{\mathrm{gr}}\) is the graded pairwise interaction package extracted here from the MTT interaction profunctors.
\end{enumerate}
\end{proposition}

\begin{proof}
The first item is the state-data output of \cite{RahmanQuiverDataI}.  The second is the support-level output of \cite{RahmanQuiverDataII}.  The third follows from Theorem~\ref{thm:graded-interaction-theorem}, which constructs
\[
I_\Sigma^{\mathrm{gr}}
=
(P_{ij}(q))_{1\le i,j\le r}
\]
from the probe evaluations of the MTT interaction profunctors.  Together these form the finite package inherited by the next stage of the sequence.
\end{proof}

\subsection{Implications for the sequel}

The inherited package
$\bigl(A_\Sigma,\ I_\Sigma^{(0/1)},\ I_\Sigma^{\mathrm{gr}}\bigr)$
is the finite algebraic input for the later stability and BPS analysis.  The present paper does not yet introduce the dynamical structures needed for that analysis.  Rather, it supplies the state, support, and graded interaction data on which those structures may be imposed.

The next stage is expected to add:
\begin{enumerate}
\item a finite charge sector or charge lattice built from the nodewise state data;
\item stability data compatible with \(c_\Sigma\) and \(I_\Sigma^{\mathrm{gr}}\);
\item an admissible interaction law using the graded pairwise matrix;
\item chamber and wall structure;
\item first BPS indices or wall-crossing shadows attached to the finite-node package.
\end{enumerate}

Thus the present paper remains a bridge paper.  It does not compute BPS spectra or wall-crossing automorphisms.  Its role is to supply the graded interaction layer required before those genuinely dynamical questions can be formulated.

%----------------------------------------------------
\section{Examples and toy models}
\label{sec:examples-toy-models}

This section illustrates the formal behavior of the MTT package.  The examples are schematic: they are not classifications of finite-node conifold degenerations, but controlled models showing how the definitions behave in limiting cases.

\subsection{Diagonal and off-diagonal behavior}

For each ordered pair \((i,j)\), the pairwise interaction complex is
\[
\mathsf H_{ij}
=
\RHom_{\mathcal C_{p_j}}(\Psi_j\Phi_i(L_i),L_j).
\]
When \(i=j\), this becomes
\[
\mathsf H_{ii}
=
\RHom_{\mathcal C_{p_i}}(\Psi_i\Phi_i(L_i),L_i),
\]
so the source and target localized sectors coincide.  When \(i\neq j\), one obtains
\[
\mathsf H_{ij}
=
\RHom_{\mathcal C_{p_j}}(\Psi_j\Phi_i(L_i),L_j),
\]
which measures transport from the \(i\)-th localized sector through the bulk into the \(j\)-th localized sector.

This distinction reflects the obstruction result of Section~\ref{sec:extension-only-obstruction}.  The shadow-side extension classes are local and therefore naturally overlap with diagonal or self-interaction behavior.  Off-diagonal terms, by contrast, require the transport composite
\[
\Psi_j\Phi_i:
\mathcal C_{p_i}\to\mathcal C_{p_j}.
\]
They are not explained by the point-supported shadow objects \(Q_i\) alone.

At the graded level, this gives two types of entries:
\[
P_{ii}(q)
=
\sum_m \dim H^m(\mathsf H_{ii})q^m,
\qquad
P_{ij}(q)
=
\sum_m \dim H^m(\mathsf H_{ij})q^m
\quad (i\neq j).
\]
The first records bulk-mediated return to the same localized sector; the second records transport between distinct localized sectors.

\begin{remark}
\label{rem:diagonal-offdiagonal}
No relation between diagonal and off-diagonal polynomials is imposed here.  In particular, the formalism does not assume that diagonal terms dominate, nor that off-diagonal terms vanish except in special geometric situations.  Such questions belong to the later stability and BPS analysis.
\end{remark}

\subsection{Directedness}

The MTT package is naturally directed.  For ordered pairs \((i,j)\) and \((j,i)\),
\[
\mathsf H_{ij}
=
\RHom_{\mathcal C_{p_j}}(\Psi_j\Phi_i(L_i),L_j),
\qquad
\mathsf H_{ji}
=
\RHom_{\mathcal C_{p_i}}(\Psi_i\Phi_j(L_j),L_i).
\]
These complexes are formed in different target categories and use different transport composites.  Hence there is no formal reason to expect
\[
\mathsf H_{ij}\cong \mathsf H_{ji},
\qquad
P_{ij}(q)=P_{ji}(q),
\qquad
I_\Sigma^{\mathrm{gr}} \text{ symmetric}.
\]

\begin{proposition}[formal directedness]
\label{prop:formal-directedness}
Within the MTT formalism, the graded interaction package
\[
I_\Sigma^{\mathrm{gr}}=(P_{ij}(q))
\]
is attached to ordered pairs of nodes.  No symmetry relation
\[
P_{ij}(q)=P_{ji}(q)
\]
is forced by the definitions.
\end{proposition}

\begin{proof}
The entry \(P_{ij}(q)\) is computed from
\[
\mathsf H_{ij}
=
\RHom_{\mathcal C_{p_j}}(\Psi_j\Phi_i(L_i),L_j),
\]
whereas \(P_{ji}(q)\) is computed from
\[
\mathsf H_{ji}
=
\RHom_{\mathcal C_{p_i}}(\Psi_i\Phi_j(L_j),L_i).
\]
The transport composites \(\Psi_j\Phi_i\) and \(\Psi_i\Phi_j\) are not identified in the MTT datum, and neither are the target categories \(\mathcal C_{p_j}\) and \(\mathcal C_{p_i}\).  Therefore the formalism supplies no canonical symmetry between the two graded entries.
\end{proof}

Thus any symmetry, skew-symmetry, or duality relation among entries of \(I_\Sigma^{\mathrm{gr}}\) would have to come from additional geometry or later stability/BPS structure, not from the MTT definitions alone.

\subsection{Triangle-visible toy models}

The following schematic examples illustrate the nonvanishing criterion of Section~\ref{sec:triangle-visible-content}.  They are abstract triangulated-category models, not asserted geometric realizations.

\begin{example}[right triangle visibility]
\label{ex:quotient-type-triangle}
Fix an ordered pair \((i,j)\) and suppose
\[
A_{ij}(L_i)=\Psi_j\Phi_i(L_i)
\]
fits into a distinguished triangle
\[
B\to A_{ij}(L_i)\xrightarrow{u}L_j\to B[1]
\]
with \(u\neq 0\).  Then \(A_{ij}(L_i)\) is right triangle-visible with respect to \(L_j\).  By Corollary~\ref{cor:mediated-triangle-visible-nonvanishing},
\[
\mathsf H_{ij}
=
\RHom_{\mathcal C_{p_j}}(A_{ij}(L_i),L_j)
\not\simeq 0.
\]
Under the finiteness convention of Convention~\ref{conv:mtt-finiteness}, this gives
\[
P_{ij}(q)\neq 0.
\]
\end{example}

Right triangle visibility is the variance directly used by the primary interaction complex \(\mathsf H_{ij}\).

\begin{example}[left triangle visibility]
\label{ex:subobject-type-triangle}
Fix an ordered pair \((i,j)\) and suppose instead that
\[
L_j\xrightarrow{v}A_{ij}(L_i)\to C\to L_j[1]
\]
is a distinguished triangle with \(v\neq 0\).  Then Proposition~\ref{prop:triangle-visible-nonvanishing} gives
\[
\RHom_{\mathcal C_{p_j}}(L_j,A_{ij}(L_i))\not\simeq 0.
\]
This detects mediated content in the opposite variance.  It does not, by itself, imply nonvanishing of the primary complex
\[
\mathsf H_{ij}
=
\RHom_{\mathcal C_{p_j}}(A_{ij}(L_i),L_j),
\]
unless additional duality or symmetry assumptions are imposed.
\end{example}

\begin{remark}
\label{rem:triangle-visibility-schematic}
These toy models only illustrate the logic of triangle-visible nonvanishing.  The paper does not claim that every supported channel admits such a triangle directly.  The conditional bridge theorem instead uses the weaker content-and-detector route of Section~\ref{sec:probe-content-detector}.
\end{remark}

\subsection{Specialization examples}

We now illustrate the scalar specializations
\[
w_{ij}^{\mathrm{tot}}=P_{ij}(1),
\qquad
w_{ij}^{\chi}=P_{ij}(-1)
\]
from Definition~\ref{def:specialized-interaction-weights}.  These examples assume the bounded Hom-finite case.

\begin{example}[single-degree interaction]
\label{ex:single-degree-interaction}
Suppose \(\mathsf H_{ij}\) has cohomology concentrated in one degree \(m_0\), with
\[
\dim H^{m_0}(\mathsf H_{ij})=d,
\qquad
d\ge 1.
\]
Then
\[
P_{ij}(q)=d\,q^{m_0},
\qquad
w_{ij}^{\mathrm{tot}}=d,
\qquad
w_{ij}^{\chi}=(-1)^{m_0}d.
\]
The total specialization records size; the Euler specialization records parity-sensitive signed size.
\end{example}

\begin{example}[two-degree interaction]
\label{ex:two-degree-interaction}
Suppose \(\mathsf H_{ij}\) has nonzero cohomology only in degrees \(m\) and \(m+1\), with
\[
\dim H^m(\mathsf H_{ij})=a,
\qquad
\dim H^{m+1}(\mathsf H_{ij})=b.
\]
Then
\[
P_{ij}(q)=a q^m+b q^{m+1},
\]
while
\[
w_{ij}^{\mathrm{tot}}=a+b,
\qquad
w_{ij}^{\chi}=(-1)^m(a-b).
\]
Thus \(P_{ij}(q)\) retains the degree distribution, while \(w_{ij}^{\mathrm{tot}}\) and \(w_{ij}^{\chi}\) collapse it to scalar shadows.
\end{example}

\begin{remark}
\label{rem:specialization-toy-models}
The specialization examples do not enter the proof of the main theorem stack.  They only show how the graded interaction polynomial interpolates between a derived interaction complex and its simplest scalar decategorifications.
\end{remark}

%----------------------------------------------------
\section{Conclusion}
\label{sec:conclusion}

\subsection{Summary of the construction}

This paper constructs the first graded pairwise interaction layer in the finite-node sequence developed in \cite{RahmanQuiverDataI,RahmanQuiverDataII}.  \cite{RahmanQuiverDataI} extracted the state-data package
$A_\Sigma=(V_\Sigma,E_\Sigma,c_\Sigma)$, and \cite{RahmanQuiverDataII} constructed the support-level interaction package $I_\Sigma^{(0/1)}$. The present paper refines that binary support law to a graded interaction package $I_\Sigma^{\mathrm{gr}}$.

The starting point was the limitation of support-level incidence.  The matrix \(I_\Sigma^{(0/1)}\) records whether a mediated channel is present, but not its derived size, cohomological degree, or exact-triangle behavior.  Section~\ref{sec:extension-only-obstruction} showed that the naive extension-only shadow route remains local and does not by itself supply the required off-diagonal transport mechanism.  This motivates the passage to the bulk/localized schober transport composites
\[
\Psi_j\Phi_i:
\mathcal C_{p_i}\to\mathcal C_{p_j}.
\]

The main construction is mediated triangle transport.  An MTT datum packages the bulk triangulated category, localized triangulated sectors, exact transport functors, localized probes, shadow compatibility, and the derived interaction profunctors
\[
\mathbb T_{ij}(X,Y)
=
\RHom_{\mathcal C_{p_j}}(\Psi_j\Phi_i(X),Y).
\]
Evaluating these profunctors on probes gives the pairwise interaction complexes
\[
\mathsf H_{ij}
=
\RHom_{\mathcal C_{p_j}}(\Psi_j\Phi_i(L_i),L_j).
\]
These complexes are the derived pairwise objects replacing support-only interaction.

The exactness results show that MTT preserves distinguished-triangle structure and yields long exact interaction sequences.  The triangle-visibility criterion gives a direct sufficient condition for nonvanishing, while the conditional bridge theorem shows that supported channels yield nonzero interaction complexes under the probe, content, and detector hypotheses.

Finally, under the standing finiteness convention, the cohomology of \(\mathsf H_{ij}\) defines
\[
P_{ij}(q)
=
\sum_m \dim H^m(\mathsf H_{ij})q^m,
\]
and these polynomials assemble into
\[
I_\Sigma^{\mathrm{gr}}
=
(P_{ij}(q)).
\]
Thus the finite output of the sequence through the present paper is
\[
(A_\Sigma,\ I_\Sigma^{(0/1)},\ I_\Sigma^{\mathrm{gr}}).
\]

\subsection{What is deferred}

The paper does not construct the later dynamical theory.  In particular, it does not introduce a complete stability condition, charge-wall chamber package, BPS index, or wall-crossing formula.  Those structures require additional choices and additional theorem-level input beyond the bridge constructed here.

The sequel is expected to add:
\begin{enumerate}
\item a charge lattice or finite charge sector associated with \(A_\Sigma\);
\item stability data on that charge sector;
\item an admissible interaction law using \(I_\Sigma^{(0/1)}\) and \(I_\Sigma^{\mathrm{gr}}\);
\item chamber and wall structure;
\item first BPS indices or wall-crossing shadows attached to the finite-node geometry.
\end{enumerate}

This deferral is deliberate.  The present paper supplies the finite state, support, and graded interaction data that the later stability/BPS formalism should inherit.

\subsection{Outlook}

The package
\[
(A_\Sigma,\ I_\Sigma^{(0/1)},\ I_\Sigma^{\mathrm{gr}})
\]
is the finite algebraic input for the next stage of the program.  It records, respectively, nodewise state data, support-level interaction, and graded pairwise interaction.  The next step is to equip this package with charge, stability, chamber, and wall-crossing structure.

Thus the finite-node conifold geometry does not stop at local extension data or binary support.  Once the bulk/localized transport system is included, it yields a derived pairwise interaction layer and hence a graded interaction matrix.  This is the final bridge needed before formulating the first finite-node stability and BPS theory.

%-----------------------------------
%               Appendix

\appendix

\section{Why binary support is too coarse for later stability and BPS theory}
\label{app:why-binary-support-too-coarse}

This appendix explains why the support-level incidence package of \cite{RahmanQuiverDataII} is necessary but not sufficient for a later stability or BPS formalism.  Binary incidence is the correct first decategorified support layer: it records where mediated channels are present.  What it does not record is the graded, homological, or exact-triangle behavior of those channels.

\subsection{Binary incidence as support-level decategorification}

In \cite{RahmanQuiverDataII}, the finite-node schober datum determines a support-level interaction package with binary decategorification
\[
I_\Sigma^{(0/1)}.
\]
This matrix records whether a mediated interaction channel is present.  It therefore captures the support of the pairwise transport pattern determined by the bulk/localized functors, but it does not attach graded multiplicity, higher-degree interaction data, or exact-triangle transport data.

Thus \(I_\Sigma^{(0/1)}\) answers the question of channel presence.  It does not yet answer the stronger questions of interaction size, cohomological degree, or behavior under extensions and decompositions.  In this sense, the binary incidence package is the support of a later interaction law, but not the full law itself.

\subsection{What binary support forgets}

The matrix \(I_\Sigma^{(0/1)}\) deliberately forgets several kinds of information.

First, it forgets graded size.  Two mediated channels may both have binary value \(1\), even if one carries a small interaction complex and the other carries a richer derived object.

Second, it forgets homological degree.  A later stability or BPS theory should distinguish a channel concentrated in degree \(0\) from one spread across several cohomological degrees.  Binary support cannot make that distinction.

Third, it forgets exact-triangle transport.  The finite-node schober package contains exact functors between localized sectors and the bulk sector.  After passing only to binary support, one loses the long exact interaction sequences produced by distinguished triangles.

Fourth, it forgets extension-visible transport phenomena.  The corrected finite-node extension package of \cite{RahmanPerverseNearbyCycles,RahmanMixedHodgeModules} shows that local sectors are coupled through a common global extension architecture.  Binary support records where transport is allowed, but not how such transport appears through cones, extensions, or derived comparison objects.

For these reasons, binary support is indispensable as a support skeleton but incomplete as an interaction law.

\subsection{Why later BPS theory needs more than support}

A later stability or BPS theory is expected to require at least four ingredients:
\begin{enumerate}
\item a finite charge sector or charge lattice built from the finite-node state data;
\item an admissibility law governing which sectors may interact;
\item a pairwise interaction law distinguishing channels beyond mere presence;
\item a mechanism controlling how interaction behaves under decomposition, extension, or chamber change.
\end{enumerate}

The support-level package of \cite{RahmanQuiverDataII} supplies the second ingredient: it determines the support of admissible interaction.  It does not supply the third or fourth.  In particular, wall-crossing data cannot be recovered from support alone, because support does not determine graded interaction size, derived visibility, or exact transport behavior.

Accordingly, the relationship should be understood as a refinement:
\[
I_\Sigma^{(0/1)}
\quad\leadsto\quad
I_\Sigma^{\mathrm{gr}}.
\]
The binary package specifies where interaction is allowed.  The graded package records the cohomological size and degree distribution of the corresponding interaction.

\subsection{Why derived pairwise objects are introduced}

The present paper replaces support-only interaction by derived pairwise interaction objects.  Once localized probes
\[
L_i\in\mathcal C_{p_i}
\]
are fixed, the mediated transport from node \(i\) to node \(j\) is encoded by
\[
\mathsf H_{ij}
=
\RHom_{\mathcal C_{p_j}}(\Psi_j\Phi_i(L_i),L_j).
\]
This object retains the information lost by binary support:
\begin{itemize}
\item it is pairwise and directed;
\item it records all cohomological degrees;
\item it is compatible with exact triangles through mediated triangle transport;
\item it admits graded shadows such as
\[
P_{ij}(q)
=
\sum_m \dim H^m(\mathsf H_{ij})q^m.
\]
\end{itemize}

The graded interaction matrix
\[
I_\Sigma^{\mathrm{gr}}
=
(P_{ij}(q))
\]
is therefore the first non-binary refinement of the support law of \cite{RahmanQuiverDataII}.  It does not replace binary support.  It upgrades it to the first level at which later stability and BPS constructions can use graded pairwise interaction data.

\subsection{Conceptual summary}

The passage from \cite{RahmanQuiverDataII} to the present paper is:
\[
\begin{tikzcd}[column sep=large]
\text{binary support}
\arrow[r]
&
\text{derived pairwise interaction}
\arrow[r]
&
\text{graded interaction shadow}.
\end{tikzcd}
\]
The first step is needed because binary support is too coarse to encode graded or exact interaction data.  The second step is needed because later stability and BPS theory require a finite pairwise interaction law richer than mere support.

Thus the support-level incidence package of \cite{RahmanQuiverDataII} is the correct first floor, but not the terminal interaction object.  The present paper supplies the next floor: the first derived and graded pairwise interaction layer that a later stability and wall-crossing theory can inherit.

%-----------------------------------------

\section{A brief note on profunctors and derived interaction profunctors}
\label{app:profunctors}

This appendix fixes the profunctor notation used in the main text.  The classical notion of a profunctor, also called a distributor, goes back to Bénabou \cite{BenabouProfunctors}.  We only use the elementary variance pattern and its derived analogue.

\subsection{Basic definition}

Let \(\mathcal A\) and \(\mathcal B\) be categories.  A \emph{profunctor} from \(\mathcal A\) to \(\mathcal B\) is a functor
\[
P:
\mathcal A^{op}\times\mathcal B
\longrightarrow
\mathbf{Set}.
\]
Thus \(P\) assigns to each pair
\[
A\in\mathcal A,
\qquad
B\in\mathcal B
\]
a set \(P(A,B)\) of generalized morphisms from \(A\) to \(B\).  In enriched or linear settings, the target \(\mathbf{Set}\) may be replaced by another category \(\mathcal V\).  In the present paper, the relevant target is the derived linear category \(D(\mathrm{Vect})\).

\subsection{The profunctor associated to a functor}

If
\[
F:\mathcal A\to\mathcal B
\]
is a functor, then
\[
(A,B)
\longmapsto
\Hom_{\mathcal B}(F(A),B)
\]
defines a profunctor from \(\mathcal A\) to \(\mathcal B\).  A morphism \(A'\to A\) gives, by precomposition, a map
\[
\Hom_{\mathcal B}(F(A),B)
\longrightarrow
\Hom_{\mathcal B}(F(A'),B),
\]
so the construction is contravariant in the \(\mathcal A\)-variable.  A morphism \(B\to B'\) gives, by postcomposition, a map
\[
\Hom_{\mathcal B}(F(A),B)
\longrightarrow
\Hom_{\mathcal B}(F(A),B'),
\]
so it is covariant in the \(\mathcal B\)-variable.

\subsection{Derived interaction profunctors}

In a triangulated or stable categorical setting, Hom-sets are replaced by derived Hom objects.  If
\[
F:\mathcal A\to\mathcal B
\]
is an exact functor between triangulated categories, the corresponding derived assignment is
\[
(A,B)
\longmapsto
\RHom_{\mathcal B}(F(A),B).
\]
This has the same variance pattern as the ordinary profunctor: contravariant in \(A\) and covariant in \(B\).  The target is a derived linear category such as \(D(\mathrm{Vect})\).  Exactness of \(F\), together with exactness properties of derived Hom, is what produces the distinguished-triangle behavior used in Section~\ref{sec:exactness-triangle-transport}; see also standard references on triangulated and stable categorical exactness \cite{KS,DimcaSheaves,CDW}.

\subsection{Specialization to mediated triangle transport}

In the main text, the finite-node schober datum supplies exact transport functors
\[
\Phi_i:
\mathcal C_{p_i}
\to
\mathcal C_{\mathrm{bulk}},
\qquad
\Psi_j:
\mathcal C_{\mathrm{bulk}}
\to
\mathcal C_{p_j}.
\]
Their composite is
\[
A_{ij}
:=
\Psi_j\Phi_i:
\mathcal C_{p_i}
\to
\mathcal C_{p_j}.
\]
The associated MTT interaction profunctor is
\[
\mathbb T_{ij}(X,Y)
:=
\RHom_{\mathcal C_{p_j}}
\bigl(A_{ij}(X),Y\bigr),
\]
or equivalently
\[
\mathbb T_{ij}:
\mathcal C_{p_i}^{op}
\times
\mathcal C_{p_j}
\to
D(\mathrm{Vect}).
\]
Thus \(\mathbb T_{ij}\) is exactly the derived profunctor associated to the exact transport functor \(A_{ij}\).  The new point in this paper is the use of such derived profunctors to encode bulk-mediated pairwise interaction between localized sectors of a finite-node conifold degeneration.

\subsection{Probe evaluation}

For localized probes
\[
L_i\in\mathcal C_{p_i},
\qquad
L_j\in\mathcal C_{p_j},
\]
the probe evaluation is
\[
\mathsf H_{ij}
:=
\mathbb T_{ij}(L_i,L_j)
=
\RHom_{\mathcal C_{p_j}}
\bigl(\Psi_j\Phi_i(L_i),L_j\bigr).
\]
This is the derived pairwise interaction object used throughout the paper.  Under the finiteness convention of Convention~\ref{conv:mtt-finiteness}, its graded shadow is
\[
P_{ij}(q)
=
\sum_m
\dim H^m(\mathsf H_{ij})q^m.
\]

\subsection{Summary}

For the main text, the only points needed are:
\begin{enumerate}
\item a profunctor from \(\mathcal A\) to \(\mathcal B\) is a bifunctor
\[
\mathcal A^{op}\times\mathcal B\to\mathcal V;
\]
\item a functor \(F:\mathcal A\to\mathcal B\) gives the profunctor
\[
(A,B)
\longmapsto
\Hom_{\mathcal B}(F(A),B);
\]
\item in the triangulated setting, this becomes
\[
(A,B)
\longmapsto
\RHom_{\mathcal B}(F(A),B);
\]
\item the MTT interaction profunctors are obtained by taking
\[
F=A_{ij}=\Psi_j\Phi_i.
\]
\end{enumerate}

%=================================================
%
%                    BIBLIOGRAPHY
%

\printbibliography

\end{document}